\documentclass[11pt]{amsart}

\usepackage{amsmath,amssymb,amsthm}
\usepackage{fullpage}
\usepackage[all]{xy}
\usepackage{color}

\providecommand{\abs}[1]{\left\lvert#1\right\rvert}

\DeclareMathOperator{\Hom}{Hom}
\DeclareMathOperator{\Char}{char}

\DeclareMathOperator{\PGL}{PGL}

\DeclareMathOperator{\Aut}{Aut}
\DeclareMathOperator{\Res}{Res}\DeclareMathOperator{\Disc}{Disc}
 \DeclareMathOperator{\End}{End}
\DeclareMathOperator{\GL}{GL}

\newcommand{\tth}{^{\operatorname{th}}}

\theoremstyle{plain}
\newtheorem{thm}{Theorem}
\newtheorem*{thm*}{Theorem}
\newtheorem{lem}[thm]{Lemma}
\newtheorem{prop}[thm]{Proposition}
\newtheorem{cor}[thm]{Corollary}
\theoremstyle{definition}
\newtheorem{defn}[thm]{Definition}

\newtheorem{exmp}[thm]{Example}
\newtheorem*{rem}{Remark}
\theoremstyle{remark}

\def\Z{\mathbb{Z}}
\def\Q{\mathbb{Q}}
\def\P{\mathbb{P}}
\def\A{\mathbb{A}}

\def\C{\mathbb{C}}

\begin{document}
    \title{Almost Newton, sometimes Latt\`es}
    \author{Benjamin Hutz and Lucien Szpiro}

    \thanks{The two authors received funding from the NSF Grants DMS-0854746 and DMS-0739346. The second author was also partially funded by PSC Cuny grant EG 36276.}



    \maketitle

\begin{abstract}
    Self-maps everywhere defined on the projective space $\P^N$ over a number field or a function field  are the basic objects of study in the arithmetic of dynamical systems. One reason is a theorem of Fakkruddin \cite{Fakhruddin} (with complements in \cite{Bhatnagar}) that asserts that a ``polarized'' self-map of a projective variety is essentially  the restriction of a self-map of the projective space given by the polarization. In this paper we study the natural self-maps defined the following way: $F$ is a homogeneous polynomial of degree $d$ in $(N+1)$ variables $X_i$ defining a smooth hypersurface. Suppose the characteristic of the field does not divide $d$ and define the map of partial derivatives $\phi_F = (F_{X_0},\ldots,F_{X_N})$.  The map $\phi_F$ is defined everywhere due to the following formula of Euler: $\sum X_i F_{X_i} = d F$,
    which implies that a point where all the partial derivatives vanish is a non-smooth point of the hypersuface $F=0$. One can also compose such a map with an element of $\PGL_{N+1}$.   In the particular case addressed in this article, $N=1$, the smoothness condition means that $F$ has only simple zeroes. In this manner, fixed points and their multipliers are easy to describe and, moreover, with a few modifications we recover classical dynamical systems like the Newton method for finding roots of polynomials or the Latt\`es map corresponding to the multiplication by $2$ on an elliptic curve.
\end{abstract}

\section{Introduction}
    Self-maps everywhere defined on the projective space $\P^N$ over a number field or a function field  are the basic objects of study in the arithmetic of dynamical systems. One reason is a theorem of Fakkruddin \cite{Fakhruddin} (with complements in \cite{Bhatnagar}) that asserts that a ``polarized'' self-map of a projective variety is essentially the restriction of a self-map of the projective space given by the polarization. In this paper we study the natural self-maps defined the following way: $F$ is a homogeneous polynomial of degree $d$ in $(N+1)$ variables $X_i$ defining a smooth hypersurface. Suppose the characteristic of the field does not divide $d$ and define the map of partial derivatives $\phi_F = (F_{X_0},\ldots,F_{X_N})$.  The map $\phi_F$ is defined everywhere due to the following formula of Euler:
    \begin{equation*}
        \sum X_i F_{X_i} = d F,
    \end{equation*}
    which implies that a point where all the partial derivatives vanish is a non-smooth point of the hypersuface $F=0$. One can also compose such a map with an element of $\PGL_{N+1}$.   In the particular case addressed in this article, $N=1$, the smoothness condition means that $F$ has only simple zeroes. In this manner, fixed points and their multipliers are easy to describe and, moreover, with a few modifications we recover classical dynamical systems like the Newton method for finding roots of polynomials or the Latt\`es map corresponding to the multiplication by $2$ on an elliptic curve.  We begin by recalling some of the definitions and objects we need from dynamical systems before stating the main results.
    
    Given a morphism $\phi:\P^1 \to \P^1$ we can iterate $\phi$ to create a (discrete) dynamical system.  We denote the $n\tth$ iterate of $\phi$ as $\phi^n = \phi(\phi^{n-1})$. Calculus students are exposed to dynamical systems through the iterated root finding method known as Newton's Method where, given a differentiable function $f(x)$ and an initial point $x_0$, one constructs the sequence
    \begin{equation*}
        x_{n+1} =\phi(x_n) = x_n - \frac{f(x_n)}{f'(x_n)}.
    \end{equation*}
    In general, this sequence converges to a root of $f(x)$.  In terms of dynamics, we would say that the roots of $f(x)$ are attracting fixed points of $\phi(x)$.  More generally, one says that $P$ is a \emph{periodic point of period $n$} for $\phi$ if $\phi^n(P) = P$.

    A common example of a dynamical system with periodic points is to take an endomorphism of an elliptic curve $[m]:E \to E$ and project onto the first coordinate.  This construction induces a map on $\P^1$ called a \emph{Latt\`es map}, and for $m \in \Z$ its degree is $m^2$ and its periodic points are the torsion points of the elliptic curve.

    Denote $\Hom_d$ as the set of degree $d$ morphisms on $\P^1$.  There is a natural action on $\P^1$ by $\PGL_{2}$ through conjugation that induces an action on $\Hom_d$.  We take the quotient as $M_d = \Hom_d/\PGL_{2}$.  By \cite{Silverman9}, the moduli space $M_d$ is a geometric quotient.  We say that $\gamma \in \PGL_{2}$ is an \emph{automorphism} of $\phi \in \Hom_d$ if $\gamma^{-1} \circ \phi \circ \gamma = \phi$.  We denote the (finite \cite{petsche}) group of automorphisms as $\Aut(\phi)$.

    Let $K$ be a number field and $F\in K[X,Y]$ be a homogeneous polynomial of degree $d$ with distinct roots.  Define
    \begin{equation*}
        \phi_F(X,Y) = [F_Y,-F_X]:\P^1 \to \P^1.
    \end{equation*}
    \begin{rem}
        We can think of $\phi_F$ as the linear combination
        \begin{equation*}
            \begin{pmatrix} 1 & 0 \\ 0 & -1\end{pmatrix} \begin{pmatrix} F_X \\ F_Y \end{pmatrix}.
        \end{equation*}
        It may also be interesting to study other families of linear combinations of the partial derivatives arising from other elements of $\GL_2$, but we do not address them in this article.
    \end{rem}
    In Section \ref{sect_almost_newton} we examine the dynamical properties of these maps.
    \begin{thm*}[Theorem \ref{thm_fixed}]
        The fixed points of $\phi_F(X,Y)$ are the solutions to $F(X,Y)=0$, and the multipliers of the fixed points are $1-d$.
    \end{thm*}
    \begin{thm*}[Theorem \ref{thm_family} and Corollary \ref{cor_family}]
        The family of maps of the form $\phi_F=(F_Y,-F_X):\P^1 \to \P^1$ is invariant under the conjugation action by $\PGL_2$.
    \end{thm*}
    We also give a description of the higher order periodic points and a recursive definition of the polynomial whose roots are the $n$-periodic points. We examine related, more general Newton-Raphson maps and, finally, recall the connection to invariant theory and maps with automorphisms.

    In Section \ref{sect_lattes} we explore the connection with Latt\`es maps.
    \begin{thm*}[Theorem \ref{thm_lattes}]
        Maps of the form
        \begin{equation*}
            \tilde{\phi}(x) = x - 3\frac{f(x)}{f'(x)}
        \end{equation*}
        are the Latt\`es maps from multiplication by $[2]$ and $f(x) = \prod(x-x_i)$ where $x_i$ are the $x$-coordinates of the $3$-torsion points.
    \end{thm*}
    Finally, when $E$ has complex multiplication ($m \not\in \Z$) the associated $\phi_F$ can have a non-trivial automorphism group.
    \begin{thm*}[Theorem \ref{thm_CM}]
        If $E$ has $\Aut(E) \supsetneq \Z/2\Z$ and the zeros of $F(X,Y)$ are torsion points of $E$, then an induced map $\phi_F$ has a non-trivial automorphism group.
    \end{thm*}

\section{Almost Newton Maps}\label{sect_almost_newton}
    Let $K$ be a field and consider a two variable homogeneous polynomial $F(X,Y) \in K[X,Y]$ of degree $d$ with no multiple roots.  Consider the degree $d-1$ map
    \begin{align*}
        \phi_F&:\P^1 \to \P^1 \\
        (X,Y) &\mapsto (F_Y(X,Y), -F_X(X,Y)).
    \end{align*}\
    In particular, $F_X = F_Y =0$ has no nonzero solutions and so $\phi_F$ is a morphism.  Label $x = \frac{X}{Y}$ and consider
    \begin{equation*}
        f(x) = \frac{F(X,Y)}{Y^d}
    \end{equation*}
    and notice that
    \begin{equation*}
        f'(x) = \frac{F_X(X,Y)}{Y^{d-1}}.
    \end{equation*}
    \begin{lem} \label{lem_newton}
        The map induced on affine space by $\phi_F$ is given by
        \begin{equation*}
            \tilde{\phi}_F(x) = x- d\frac{f(x)}{f'(x)}.
        \end{equation*}
    \end{lem}
    \begin{proof}
        \begin{equation*}
             \tilde{\phi}_F(x) = -\frac{F_Y(X,Y)}{F_X(X,Y)} = -\frac{YF_Y(X,Y)}{YF_X(X,Y)} = \frac{XF_X(X,Y) - dF(X,Y)}{YF_X(X,Y)} = x-d\frac{f(x)}{f'(x)}.
        \end{equation*}
    \end{proof}
    \begin{defn}
        Let $\phi = (\phi_1,\phi_2):\P^1 \to \P^1$ be a rational map on $\P^1$. Define $\Res(\phi) = \Res(\phi_1,\phi_2)$, the \emph{resultant} of the coordinate functions of $\phi$. For a homogeneous polynomial $F$, denote $\Disc(F)$ for the \emph{discriminant} of $F$.
    \end{defn}
    \begin{prop}\label{prop_res_disc}
        Let $F(X,Y)$ be a homogeneous polynomial of degree $d$ with no multiple roots.  Then,
        \begin{equation*}
            \Res(\phi_F(X,Y)) = (-1)^{d(d-1)/2}d^{d-2}\Disc(F(X,Y)).
        \end{equation*}
    \end{prop}
    \begin{proof}
        Denote $F(X,Y) = a_dX^d + a_{d-1}X^{d-1}Y + \cdots + a_0Y^d$. Then we have
        \begin{align*}
            F_X(X,Y) &= da_dX^{d-1} + \cdots + a_1Y^{d-1}\\
            F_Y(X,Y) &= a_{d-1}X^{d-1} + \cdots +da_0Y^{d-1}.
        \end{align*}
        From standard properties of resultants and discriminants we have
        \begin{align*}
            a_d\Disc(F(X,Y)) &= (-1)^{d(d-1)/2}\Res(F(X,Y),F_X(X,Y))\\
              &= (-1)^{d(d-1)/2}\frac{(-1)^{d}}{d^{d-1}}\Res(dF(X,Y),-F_X(X,Y))\\
              &= (-1)^{d(d-1)/2}\frac{(-1)^{d}}{d^{d-1}}\Res(XF_X(X,Y)+ YF_Y(X,Y),-F_X(X,Y))\\
              &= (-1)^{d(d-1)/2}\frac{(-1)^{d}}{d^{d-1}}\Res(YF_Y(X,Y),-F_X(X,Y)).
        \end{align*}
        Now we see that
        \begin{equation*}
            \Res(YF_Y,-F_X) = \abs{\begin{matrix}
              0 & a_{d-1} & 2a_{d-2} & \cdots & da_1 & 0\\
              0 & 0 & a_{d-1} & 2a_{d-2} & \cdots & da_1\\
              \vdots & & & & \vdots& \\
              -da_d & -(d-1)a_{d-1} & \cdots & -a_1 & 0 & 0\\
              0 & -da_d & -(d-1)a_{d-1} & \cdots & -a_1 & 0\\
              \vdots & & & & \vdots& \\
            \end{matrix}}.
        \end{equation*}
        Expanding down the first column we have
        \begin{align*}
            \Res&(YF_Y(X,Y),-F_X(X,Y))\\
             &= -da_n(-1)^{d+1} \abs{\begin{matrix}
              a_{d-1} & 2a_{d-2} & \cdots & da_1 & 0 &0\\
              0 & a_{d-1} & 2a_{d-2} & \cdots & da_1 &0\\
              \vdots & & & & \vdots &\\
              -da_d &-(d-1)a_{d-1} & \cdots & -a_1 & 0 & 0\\
               0 & -da_d &-(d-1)a_{d-1} & \cdots & -a_1 & 0\\
              \vdots & & & & \vdots &\\
            \end{matrix}}\\
            &= da_d(-1)^{d+2} R(F_Y(X,Y),-F_X(X,Y)).
        \end{align*}
        Thus, we compute
        \begin{align*}
            a_d\Disc(F(X,Y)) &= (-1)^{d(d-1)/2}\frac{(-1)^{d}}{d^{d-1}}\Res(YF_Y(X,Y),-F_X(X,Y))\\
              &= (-1)^{d(d-1)/2}\frac{(-1)^{d}}{d^{d-1}}(-1)^{d+2}da_n\Res(F_Y(X,Y),-F_X(X,Y))\\
              &= (-1)^{d(d-1)/2}\frac{a_d}{d^{d-2}}\Res(F_Y(X,Y),-F_X(X,Y)).
        \end{align*}
    \end{proof}
    \begin{rem}
        The similar relationship for flexible Latt\`es maps ($[m]$ for $m \in \Z$)
        \begin{equation*}
           \Disc(\Psi_{E,m-1}\Psi_{E,m+1}) = c \Res(\phi_{E,m}),
        \end{equation*}
        where $\Psi_{E,m}$ is the $m$-division polynomial and $\phi_{E,m}$ is the Latt\`es map induced by $[m]$, seems to not be currently known. Using conjectures on
        $\Disc(\Psi_{E,m})$ from \cite{Burhanuddin} and the formula for $\Res(\phi_{E,m})$ \cite[Exercise 6.23]{Silverman10} it appears that the exponent is correct and that constant should be
         \begin{equation*}
            c=\pm 2^a(m-1)^b(m+1)^c
         \end{equation*}
         for some integers $a,b,c$.  It would be interesting to determine the exact relation.
    \end{rem}
%
%
%
    \begin{defn}
        Let $P$ be a periodic point of period $n$ for $\tilde{\phi}$, then the \emph{multiplier} at $P$ is the value $(\tilde{\phi}^n)'(P)$.  If $P$ is the point at infinity, then we can compute the multiplier by first changing coordinates.
    \end{defn}
    \begin{thm} \label{thm_fixed}
        The fixed points of $\phi_F(X,Y)$ are the solutions to $F(X,Y)=0$, and the multipliers of the fixed points are $1-d$.
    \end{thm}
    \begin{proof}
        The projective equality
        \begin{equation*}
            \phi(X,Y) = (X,Y)
        \end{equation*}
        is equivalent to
        \begin{equation*}
            YF_Y(X,Y) = -XF_X(X,Y).
        \end{equation*}
        Using the formula of Euler for homogeneous polynomials we then have
        \begin{equation*}
            XF_X(X,Y) + YF_Y(X,Y) = dF(X,Y) = 0.
        \end{equation*}
        Since $d$ is a nonzero integer the fixed points satisfy $F(X,Y) = 0$.

        To calculate the multipliers, we first examine the affine fixed points.  We take a derivative evaluated at a fixed point to see
        \begin{equation*}
            \tilde{\phi}_F'(x) = 1 - d\frac{f'(x)f'(x)-f(x)f''(x)}{(f'(x))^2} = 1-d\frac{f'(x)f'(x)}{(f'(x))^2} = 1-d.
        \end{equation*}
        If a fixed point has multiplier one, then it would have multiplicity at least 2 and, hence, would be at least a double root of $F$.  Since $F$ has no multiple roots, every multiplier is not equal to one.  Thus, to see that the multiplier at infinity (when it is fixed) is also $1-d$ we may use the relation \cite[Theorem 1.14]{Silverman10}
        \begin{equation}\label{eq_mult}
            \sum_{i=1}^d \frac{1}{1-\lambda_i} = 1.
        \end{equation}
    \end{proof}
    \begin{rem}
        If $\Char{K} \mid d$, then $\phi_F$ is the identity map.
        Let $F(X,Y) = a_dX^d + a_{d-1}X^{d-1}Y + \cdots + a_0Y^d$.  Then we have
        \begin{align*}
            F_X(X,Y) &= (d-1)a_{d-1}X^{d-1}Y + \cdots a_1Y^{d-1} = Y((d-1)a_{d-1}X^{d-1} + \cdots a_1Y^{d-2})\\
            F_Y(X,Y) &= a_{d-1}X^{d-1} + \cdots + (d-1)a_1Y^{d-2}X = X(a_{d-1}X^{d-1} + \cdots + (d-1)a_1Y^{d-2}).
        \end{align*}
        Since $-i \equiv d-i \pmod{d}$ for $0 \leq i \leq d$ we have that
        \begin{equation*}
            \phi_F(X,Y) = (F_Y,-F_X) = (XP(X,Y),YP(X,Y)) = (X,Y),
        \end{equation*}
        where $P(X,Y)$ is a homogeneous polynomial.
    \end{rem}
    We next show that maps of the form $\phi_F$ form a family in the moduli space of dynamical systems.  In other words, for every $\gamma \in \PGL_2$ and $\phi_F$, there exists a $G(X,Y)$ such that $\gamma^{-1} \circ \phi_F \circ \gamma = \phi_{G}$.  In fact, $G(X,Y)$ is the polynomial that results from allowing $\gamma^{-1}$ to act on $F$.
    \begin{thm}\label{thm_family}
      Every rational map $\phi:\P^1 \to \P^1$ of degree $d-1$ whose fixed points are $\{(a_1,b_1),\ldots,(a_d,b_d)\}$ all with multiplier $(1-d)$ is a map of the form $\phi_F(X,Y)=(F_Y(X,Y),-F_X(X,Y))$ for
      \begin{equation*}
        F(X,Y) = (b_1X-a_1Y)(b_2X-a_2Y)\cdots (b_dX-a_dY).
      \end{equation*}
    \end{thm}
    \begin{proof}
        Let $(a_1,b_1),\ldots,(a_d,b_d)$ be the collection of fixed points for the map $\psi(X,Y): \P^1 \to \P^1$ whose multipliers are $1-d$.  Then on $\A^1$ we may write the map of degree $d-1$ as
        \begin{equation*}
            \tilde{\psi}(x) = x- \frac{P(x)}{Q(x)}
        \end{equation*}
        for some pair of polynomials $P(x)$ and $Q(x)$ with no common zeros.
        Let $\tilde{\phi}_F(x)$ be the affine map associated to $F(X,Y) = (b_1X-a_1Y)\cdots (b_dX-a_dY)$ and we can write
        \begin{equation*}
            \tilde{\phi}_F(x) = x- d\frac{f(x)}{f'(x)}
        \end{equation*}
        where
        \begin{equation*}
            f(x) = \frac{F(X,Y)}{Y^d}.
        \end{equation*}
        The fixed points of $\tilde{\psi}(x)$ are the points where $\frac{P(x)}{Q(x)} = 0$ and, hence, where $P(x) = 0$.  The fixed points of $\tilde{\psi}(x)$ are the same as for $\tilde{\phi}_F(x)$, so we must have $P(x) = cf(x)$ for some nonzero constant $c$.  Using the fact that the multipliers are $1-d$, we get
        \begin{equation*}
            \tilde{\psi}'(x) = 1 - \frac{cf'Q - cQ'}{(Q')^2} = 1 - \frac{cf'}{Q} = 1-d.
        \end{equation*}
        Therefore, we know that
        \begin{equation*}
            \frac{c}{d}f'(x_i) = Q(x_i)
        \end{equation*}
        where $x_1,\ldots,x_d$ are the fixed points (or $x_1,\ldots,x_{d-1}$ if $(1,0) \in \P^1$ is a fixed point).
        Since $f'(x)$ and $Q(x)$ are both degree $d-1$ polynomials (or $d-2$), this is a system of $d$ (or $d-1$) equations in the $d$ (or $d-1$) coefficients of $Q(x)$.  Since the values $x_i$ are distinct (since the multipliers are $\neq 1$) the Vandermonde matrix is invertible and we get a unique solution for $Q(x)$.  In particular, we must have
        \begin{equation*}
            \frac{c}{d}f'(x) = Q(x)
        \end{equation*}
        and thus
        \begin{equation*}
            \tilde{\psi}(x) = \tilde{\phi}(x).
        \end{equation*}
    \end{proof}

    \begin{cor} \label{cor_family}
        The family of maps of the form $\phi_F(X,Y)=(F_Y(X,Y),-F_X(X,Y)):\P^1 \to \P^1$ is invariant under the conjugation action by $\PGL_2$.  In particular, the family of $\phi_F$ where $\deg{F(X,Y)} = d$ is isomorphic to an arbitrary choice of $d-3$ distinct points in $\P^1 - \{0,1,\infty\}$.
    \end{cor}
    \begin{proof}
        Conjugation fixes the multipliers and moves the fixed points, so by Theorem \ref{thm_family} the conjugated map is of the same form.

        A map of degree $d-1$ on $\P^1$ has $d$ fixed points.  The action by $\PGL_2$ can move any 3 distinct points to any 3 distinct points.  Thus, the choice of the remaining $d-3$ fixed points determines $\phi_F$.
    \end{proof}
    \subsection{Extended Example}
        \begin{prop} \label{prop_deg4_form}
            Let $F(X,Y)$ be a degree $4$ homogeneous polynomial with no multiple roots with associated morphism $\phi_F(X,Y)$.  For any $\alpha \in \overline{\Q} - \{0,1\}$ we have that $\phi_F(X,Y)$ is conjugate to a map of the form
            \begin{equation*}
                \phi_{F,\alpha}(X,Y) = (X^3 - 2(\alpha+1)X^2Y + 3\alpha XY^2, -3X^2Y + 2(\alpha+1)XY^2 - \alpha Y^3).
            \end{equation*}
        \end{prop}
        \begin{proof}
            We can move three of the 4 fixed points to $\{0,1,\infty\}$ with an element of $\PGL_2$ and label the fourth fixed point as $\alpha$.  Then we have
            \begin{equation*}
                F(X,Y,\alpha) = (X)(Y)(X-Y)(X-\alpha Y) = X^3Y -(\alpha+1) X^2Y^2 + \alpha XY^3
            \end{equation*}
            and
            \begin{align*}
                \phi_{F,\alpha}(X,Y) &= (F_Y(X,Y,\alpha),-F_X(X,Y,\alpha))\\
                &= (X^3 - 2(\alpha+1)X^2Y + 3\alpha XY^2, -(3X^2Y - 2(\alpha+1)XY^2 + \alpha Y^3)).
            \end{align*}
        \end{proof}
        \begin{prop}
            Let $F(X,Y)$ be a degree $4$ homogeneous polynomial with no multiple roots with associated morphism $\phi_F(X,Y)$.  Assume that $\phi_F(X,Y)$ is in the form of Proposition \ref{prop_deg4_form}.  Then, the two periodic points are of the form
            \begin{equation*}
                \{\pm \sqrt{\alpha}, 1 \pm \sqrt{1 - \alpha}, \alpha \pm \sqrt{\alpha^2 -\alpha}\} \cup \{0,1,\infty, \alpha\}.
            \end{equation*}
        \end{prop}
        \begin{proof}
            Direct computation.
        \end{proof}
        \begin{prop}
            $\Q$-Rational affine two periodic points are parameterized by pythagorean triples.
        \end{prop}
        \begin{proof}
            The values $\alpha$ and $1-\alpha$ are both squares and $0 < \alpha < 1$.  Thus, there are relatively prime integers $p$ and $q$ so that $\alpha = \frac{p^2}{q^2}$ with $p < q$ and $1- \alpha = \frac{q^2-p^2}{q^2}$.  Therefore, $r^2 + p^2 = q^2$ is a pythagorean triple, with $r^2 = (1-\alpha)q^2$.
        \end{proof}
        \begin{rem}
            The 2-periodic points are not the roots of $f(\tilde{\phi}(x))$, see Theorem \ref{thm_higher_periodic} for the general relation.

            For general $F(X,Y)$, $\phi_{F}^2(X,Y)$ does not come from a homogeneous polynomial $G$.
        \end{rem}

    \subsection{Higher order periodic points}
        We set the following notation
        \begin{align*}
            f(x) &= \frac{F(X,Y)}{Y^d} = \sum_{i=0}^{d-1} a_ix^i\\
            \tilde{\phi}^n(x) &= \frac{A_n(x)}{B_n(x)}\\
            c_n &= -\frac{B_{n+1}(x)}{F_X(A_n(x),B_n(x))}
        \end{align*}
        where $A_n(x)$ and $B_n(x)$ are polynomials and $c_n$ is a constant.
        \begin{defn}
            Let $\Psi_n(x)$ be the polynomial whose zeros are affine $n$-periodic points.

            The polynomial $\Psi_n(x)$ is the equivalent of the $n$-division polynomial for elliptic curves, see \cite[Chapter 2]{lang3} for information on division polynomials.
        \end{defn}
         While it is possible, to define $\Psi_n(x)$ recursively, the relation is not as simple as for elliptic curves. If we let $\Psi_{E,m}$ be the $m$-division polynomial for an elliptic curve $E$, then
         \begin{align*}
            \Psi_{E,2m+1} &= \Psi_{E,m+2}\Psi_{E,m}^3 - \Psi_{E,m-1}\Psi_{E,m+1}^3 \quad \text{for } m \geq 2\\
            \Psi_{E,2m} &= \left(\frac{\Psi_{E,m}}{2y}\right)(\Psi_{E,m+2}\Psi_{E,m-1}^2 - \Psi_{E,m-2}\Psi_{E,m+1}^2) \quad \text{for } m \geq 3.
         \end{align*}
         Notice that these relations depend only on $\Psi_{E,m}$ for various $m$, whereas the formula in the following theorem also involves iterates of the map.
        \begin{thm} \label{thm_higher_periodic}
            We have the following formulas
            \begin{equation*}
                \tilde{\phi}^n(x) = x + d\frac{\Psi_n(x)}{B_n(x)}
            \end{equation*}
            and
            \begin{equation*}
                \Psi_{n+1}(x) =\frac{F(A_n(x),B_n(x)) - \Psi_n(x)F_X(A_n(x),B_n(x))}{B_n(x)c_n}
            \end{equation*}
            with multipliers
            \begin{equation*}
                \prod_{i=0}^{n-1}\left(1-d + d\frac{f(\phi^i(x))f''(\phi^i(x))}{f'(\phi^i(x))^2}\right).
            \end{equation*}
        \end{thm}
        \begin{proof}
            We proceed inductively.  For $n=1$ we know that the fixed points are the zeros of $f(x)$.
            \begin{equation*}
                \tilde{\phi}(x) = x - d\frac{f(x)}{f'(x)} = x - d\frac{f(x)}{F_X(A_0(x),B_0(x))} = x - d\frac{f(x)}{-B_1(x)} = x + d\frac{\Psi_1(x)}{B_1(x)}.
            \end{equation*}
            Now assume that
            \begin{equation*}
                 \tilde{\phi}^n(x) = x + d\frac{\Psi_n(x)}{B_n(x)}.
            \end{equation*}
            Computing
            \begin{align*}
                \tilde{\phi}^{n+1}(x) &= x + d\frac{\Psi_n(x)}{B_n(x)} - d\frac{f(\tilde{\phi}^n(x))}{f'(\tilde{\phi}^n(x))}\\
                &= x + d\frac{\Psi_n(x)}{B_n(x)} - d\frac{F(A_n(x),B_n(x))}{B_n(x)F_X(A_n(x),B_n(x))}\\
                &= x - d\frac{F(A_n(x),B_n(x)) - \Psi_n(x)F_X(A_n(x),B_n(x))} {B_n(x)F_X(A_n(x),B_n(x))}\\
                &= x - d\frac{F(A_n(x),B_n(x)) - \Psi_n(x)F_X(A_n(x),B_n(x))} {c_nB_n(x)B_{n+1}(x)}.
            \end{align*}
            So we have to show that $B_n(x)$ divides $F(A_n(x),B_n(x)) - \Psi_n(x)F_X(A_n(x),B_n(x))$.  Working modulo $B_n(x)$ we see that
            \begin{align*}
                F(A_n(x),B_n(x)) - \Psi_n(x)F_X(A_n(x),B_n(x)) \equiv A_n(x)^d - (A_n(x)/d)dA_n(x)^{d-1} \equiv 0 \pmod{B_n(x)}
            \end{align*}
            where we used the induction assumption for $\Psi_n(x)$.  Thus, the $n$-periodic points are among the roots of $\Psi_{n}(x)$.

            For equivalence, we count degrees.  Again, proceeding inductively it is clear for $n=1$. For $n+1$ we have that
            \begin{equation*}
                \deg(F(A_n(x),B_n(x)) = d(d-1)^n = (d-1)^{n+1} + (d-1)^n
            \end{equation*}
            and
            \begin{equation*}
                \deg(\Psi_n(x)F_X(A_n(x),B_n(x))) \leq  (d-1)^n + 1 + (d-1)^{n+1}
            \end{equation*}
            depending on whether the point at infinity is periodic or not.  Thus,
            \begin{equation*}
                \deg(\Psi_{n+1}(x))  \leq (d-1)^n + 1 + (d-1)^{n+1} - (d-1)^n = (d-1)^{n+1} +1.
            \end{equation*}
            Since the number of (projective) periodic points of $\phi^n$ is $(d-1)^{n}+1$, every affine fixed point must be a zero of $\Psi_n(x)$.

            We compute the multipliers as
            \begin{align*}
                \tilde{\phi}'(x) &= 1-d\frac{f'(x)^2 - f(x)f''(x)}{f'(x)^2} = 1-d + d\frac{f(x)f''(x)}{f'(x)^2}\\
                (\tilde{\phi}^n(x))' &= \prod_{i=0}^{n-1} \tilde{\phi}'(\tilde{\phi}^i(x)) = \prod_{i=0}^{n-1}\left(1-d + d\frac{f(\tilde{\phi}^i(x))f''(\tilde{\phi}^i(x))}{f'(\tilde{\phi}^i(x))^2}\right).
            \end{align*}
        \end{proof}

%
%
%
%
%
%
%

    \subsection{Replace $d$ with $r$: Modified Newton-Raphson Iteration}
        We have considered maps of the form
        \begin{equation*}
            \tilde{\phi}_F(x) = x - d\frac{f(x)}{f'(x)}
        \end{equation*}
        where $d = \deg(F(X,Y))$.  However, we could also consider affine maps of the form
        \begin{equation}\label{eq_NR}
            \tilde{\phi}(x) = x - r\frac{f(x)}{f'(x)}
        \end{equation}
        for some $r \neq 0$ and polynomial $f(x)$.  When used for iterated root finding, such maps are often called the modified Newton-Raphson method.  The fixed points are again the zeros of $f(x)$ and are all distinct with multipliers $1-r$.  Thus, if $\deg{f} \neq r$, then the point at infinity must also be a fixed point by (\ref{eq_mult}) with multiplier
        \begin{align*}
            \sum_{i=1}^{d+1} \frac{1}{1-\lambda_i} &= \frac{\deg{f(x)}}{r} + \frac{1}{1-\lambda_{\infty}} = 1\\
            \lambda_{\infty} &= \frac{\deg{f(x)}}{\deg{f(x)}-r}.
        \end{align*}
        These maps also form a family in the moduli space of dynamical systems and are determined by their fixed points..
        \begin{thm} \label{thm_modified_newton}
            Let $r$ be a non-zero integer.  Every rational map $\phi:\P^1 \to \P^1$ of degree $d-1$ which has $d-1$ affine fixed points all with multiplier $(1-r)$ and fixes $(1,0)$ with multiplier $\frac{d-1}{d-r-1}$ is a map of the form (\ref{eq_NR}).
        \end{thm}
        \begin{proof}
            The method of proof is identical to the proof of Theorem \ref{thm_family}, so is omitted.
        \end{proof}
        \begin{rem}
            Note that if we choose $r=1$, then all of the affine fixed points are also critical points ($\tilde{\phi}'(x)=0$) as noted in \cite[Corollary 1]{Crane}.
        \end{rem}

     \subsection{Connection to Maps with Automorphisms} \label{sect_automorphisms}
        Let $\Gamma \subset \PGL_2$ be a finite group.
        \begin{defn}
            We say that a homogeneous polynomial $F$ is an \emph{invariant of $\Gamma$} if $F \circ \gamma = \chi(\gamma) F$ for all $\gamma \in \Gamma$ and some character $\chi$ of $\Gamma$.  The \emph{invariant ring of $\Gamma$} denoted $K[X,Y]^{\Gamma}$ is the set of all invariants.
        \end{defn}
        The following was known as early as \cite[footnote p.345]{klein}.
        \begin{thm}\label{prop_auto}
            If $F(X,Y)$ is a homogeneous invariant of a finite group $\Gamma \subset \PGL_2$, then $\Gamma \subset \Aut(\phi_F)$.
        \end{thm}
        \begin{proof}
            Easy application of the chain rule.
        \end{proof}

    \section{Connection to Latt\`es Maps} \label{sect_lattes}
    Consider an elliptic curve with Weierstrass equation $E: y^2 = g(x)$ for $g(x) = x^3 + ax^2+bx+c$.  The solutions $g(x) = 0$ are the $2$-torsion points.  If we integrate $g(x)$ we get $G(x)=x^4/4 + a/3x^3 + b/2x^2 + cx + C$ for some constant $C$.  If we let $C = -(b^2-4ac)/12$, then the solutions $G(x)=0$ are the $3$-torsion points.

    Recall that a Latt\`es map is the induced rational function on the first coordinate of the multiplication map $[m] \in \End(E)$ on the rational points of an elliptic curve $E$; $\phi_{E,m}(x(P)) = x([m])$.
        For integers $m \geq 3 $ we have
        \begin{equation*}
            [m](x,y) = \left(x - \frac{\Psi_{E,m-1}\Psi_{E,m+1}}{\Psi_{E,m}^2},\frac{\Psi_{E,m+2}\Psi_{E,m-1}^2 - \Psi_{E,m-2}\Psi_{E,m+1}^2}{4y\Psi_{E,m}^3}\right).
        \end{equation*}
        In other words, the induced Latt\`es map is given by
        \begin{equation*}
            \phi_{E,m}(x) = x - \frac{\Psi_{E,m-1}\Psi_{E,m+1}}{\psi_{m}^2}.
        \end{equation*}
        Hence the fixed points of the Latt\`es maps are the $x$-coordinates of the $m-1$ and $m+1$ torsion points.  For $m=2$, the fixed points of $\phi_{E,2}$ are the $3$ torsion points.

        \begin{exmp}
            Given an elliptic curve of the form $y^2 = g(x) = x^3 + ax^2 + bx + c$.
            The $2$-torsion points satisfy $y^2=0$, so are fixed points of the map derived from homogenizing $g(x)$.
            \begin{align*}
                F(X,Y) &= X^3 + aX^2Y + bXY^2 + XY^3\\
                \phi_F(X,Y) &= (aX^2 + 2bXY + 3XY^2, -(2aX + bY^2))
            \end{align*}

            The fixed points of the doubling map are the points where $x([2]P) = x(P)$, in other words, the points of order $3$.  They are the points which satisfy the equation
            \begin{equation*}
                \Psi_{E,3}(x) = 3x^4 + 4ax^3 + 6bx^2 + 12cx + (4ac-b^2) = 2g(x)g''(x) - (g'(x))^2
            \end{equation*}
            So we have
            \begin{align*}
                F(X,Y) &= 3X^4 + 4aX^3Y + 6bX^2Y^2 + 12cXY^3 + (4ac-b^2)Y^4\\
                \phi_F(X,Y) &= (4aX^3 + 12bX^2Y +36cXY^2 + 4(4ac-b^2)Y^3,\\
                &-(12X^3 + 12aX^2Y + 12bXY^2 + 12cY^3)).
            \end{align*}
        \end{exmp}
        For $m=2$ we get the following stronger result, connecting generalized $\phi_F$ and Latt\`es maps.
        \begin{thm} \label{thm_lattes}
            Maps of the form
            \begin{equation*}
                \tilde{\phi}(x) = x - 3\frac{f(x)}{f'(x)}
            \end{equation*}
            are the Latt\`es maps from multiplication by $[2]$ and $f(x) = \prod(x-x_i)$ where $x_i$ are the $x$-coordinates of the $3$-torsion points.
        \end{thm}
        \begin{proof}
            From \cite[Proposition 6.52]{Silverman10} we have the multiplies are all $-2$ except at $\infty$ where it is $4$ and the fixed points are the $3$ torsion points (plus $\infty$).  Now apply Theorem \ref{thm_modified_newton}.
        \end{proof}


        \subsection{Complex Multiplication and Automorphisms}
            For an elliptic curve $E$, every automorphism is of the form $(x,y) \mapsto (u^2x,u^3y)$ for some $u \in \C^{\ast}$ \cite[III.10]{silverman1}.  In general, the only possibilities are $u = \pm 1$ and $\Aut(E) \cong \Z/2\Z$.  However, in the case of complex multiplication $\End(E) \supsetneq \Z$ and it is possible to contain additional roots of unity, thus having $\Aut(E) \supsetneq \Z/2\Z$.  The two cases are $j(E)= 0,1728$ having $\Aut(E) \cong \Z/6\Z, \Z/4\Z$ respectively \cite[III.10]{silverman1}. These additional automorphisms induce a linear action $x \mapsto u^2x$ which fixes a polynomial whose roots are torsion points.  Thus, the corresponding map $\phi_F$ has a non-trivial automorphism of the form
            \begin{equation*}
                \begin{pmatrix}
                  u^2 & 0 \\ 0 & 1
                \end{pmatrix} \in \PGL_2.
            \end{equation*}
            Thus we have shown the following theorem.

            \begin{thm}\label{thm_CM}
                If $E$ has $\Aut(E) \supsetneq \Z/2\Z$ and the zeros of $F(X,Y)$ are torsion points of $E$, then an induced map $\phi_F$ has a non-trivial automorphism group.
            \end{thm}
            \begin{exmp}
                Let $E = y^2 = x^3 + ax$, for $a \in \Z$, then $j(E)=1728$ and $\End(E)$ contains the map $(x,y) \mapsto (-x,iy)$.  Thus, the automorphism group of every $\phi_F$ coming from torsion points satisfies
                \begin{equation*}
                    \langle \begin{pmatrix} -1 & 0 \\ 0 & 1 \end{pmatrix} \rangle \subset \Aut(\phi_F).
                \end{equation*}
            \end{exmp}



\end{document}